\documentclass[12pt]{article}
\usepackage{graphicx}
\usepackage{epstopdf}
\usepackage{subfigure}
\usepackage{color}
\usepackage[svgnames]{xcolor}
\usepackage[english]{babel}
\usepackage{amsmath,amsthm}
\usepackage{array}
\usepackage{multirow}
\usepackage{amssymb}
\usepackage{isomath}
\usepackage{blkarray}
\usepackage{cancel}
\usepackage{extarrows}

\usepackage{amsfonts}
\usepackage{booktabs}
\usepackage{algorithm}
\usepackage{algorithmic}
\usepackage{pifont}
\usepackage{bm} 
\usepackage{authblk}
\usepackage{appendix}
\usepackage{mathtools}

\usepackage{caption} 
\usepackage{graphicx}
\usepackage{subfigure}
\usepackage{tabularx}
\usepackage[colorlinks=true, linkcolor=red, citecolor=green]{hyperref}
\usepackage{url}
\newtheorem{thm}{Theorem}[section]
\newtheorem{defi}{Definition}[section]

\newtheorem{lem}[thm]{Lemma}

\newtheorem{coro}{Corollary}[section]

\newtheorem{Problem}{Problem}[section]

\title{Partial Petrial Polynomials of Bouquets}

\author{Xiaoxiang Yu$^{a}$, Rong-Xia Hao$^{a}$, Jianbing Liu$^{a, }$\thanks{Corresponding author}, Zhiguo Li$^{b}$,\\
\small $^{a}$School of Mathematics and Statistics,\\
\small Beijing Jiaotong University, Beijing 100044, P. R. China\\
\small $^b$School of Science,\\
\small Hebei University of Technology, Tianjin 300401, P. R. China\\
\small Email: yuxiaoxiang1126@163.com, rxhao@bjtu.edu.cn,\\ \small jbliu1@bjtu.edu.cn, zhiguolee@hebut.edu.cn}

\date{}

\begin{document}
\baselineskip 0.65cm
\maketitle

\vspace{-2.8em}
\begin{abstract}
Gross, Mansour, and Tucker [European J. Combin., 95 (2021): 103329] introduced the \emph{partial Petrial polynomial} of a ribbon graph $G$, denoted by $^{\partial}{\varepsilon^{\times}_{G}}(z)$. 
For a prime bouquet $B_n$, Yan and Li [Discrete Appl. Math., 375 (2025): 281–289] determined $^{\partial}{\varepsilon^{\times}_{B_n}}(z)$ when the intersection graph $I(B_n)$ is either the complete graph or a path, and provided an equivalent condition under which the lowest degree of the nonzero coefficient in $^{\partial}{\varepsilon^{\times}_{B_n}}(z)$ is $1$. 
In this paper, we determine $^{\partial}{\varepsilon^{\times}_{B_n}}(z)$ when the intersection graph is a cycle. 
Moreover, we present a complete characterization of the prime bouquets whose lowest nonzero term in $^{\partial}{\varepsilon^{\times}_{B_n}}(z)$ is of degree $2$ and determine the partial Petrial polynomial for the prime bouquets. As corollaries, we determine $^{\partial}{\varepsilon^{\times}_{B_n}}(z)$ when $I(B_n)$ is the complete bipartite and tripartite graph.
\vspace{0.3em}


\end{abstract}

\section{Introduction}
Throughout this paper, all graphs considered are connected multigraphs (i.e., may contain loops and parallel edges). A {\em ribbon graph} $G$ is a surface with boundary represented as the union of two sets of closed topological discs, called vertex-discs $V(G)$ and edge-ribbons $E(G)$, satisfying the following conditions:
\begin{itemize}	
\vspace{-0.5em}
\item [(1)] the vertex-discs and edge-ribbons intersect in disjoint line segments;
\vspace{-0.5em}
\item [(2)] each such line segment lies on the boundary of precisely one vertex-disc and precisely one edge-ribbon;
\vspace{-0.5em}
\item [(3)] every edge-ribbon contains exactly two such line segments.
\end{itemize}

\vspace{-0.6em}
\noindent This definition, given by Bollob\'{a}s and Riordan \cite{2002_Ribbon graph}, provides an equivalent representation of a graph cellularly embedded in a surface.

The operation \textit{Petrial} was first introduced by Wilson in 1979 \cite{Introduction_Petrial}. Specifically, the Petrial of a ribbon graph $G$, denoted $G^{\times}$, is obtained by detaching one end of each edge from its incident vertex-disc, applying a half-twist to the edge, and reattaching it to the vertex-disc. When this operation is applied only to a subset $A$ of $E(G)$, the resulting graph  is called a \textit{partial Petrial} $G^{\times|A}$ of $G$ with respect to $A$. In 2021, Gross, Mansour, and Tucker \cite{2021_II} introduced the \textit{partial Petrial polynomial}, which enumerates all partial Petrials of a ribbon graph according to their Euler genera $\varepsilon(G^{\times|A})$. It is defined as follows.

\vspace{-0.3em}
\begin{defi}\emph{(\cite{2021_II})}\label{poly_Petrial}
The {partial Petrial polynomial} of any ribbon graph $G$ is the generating function
$$^{\partial}{\varepsilon^{\times}_{G}}(z)=\sum\limits_{A\subseteq E(G)}z^{\varepsilon{(G^{\times|A})}}$$
\noindent that enumerates partial Petrials of $G$ by Euler genus. 
\end{defi}

\vspace{-0.3em}
\noindent In \cite{2021_II}, Gross, Mansour, and Tucker  showed that the polynomial $^{\partial}{\varepsilon^{\times}_{G}}(z)$ is \emph{interpolating}, meaning that its nonzero coefficients correspond to consecutive powers of the variable. They also introduced the restricted orientable partial Petrial polynomial of $G$ by restricting the generating function to $A\subseteq E(G)$ such that $G^{\times|A}$ is orientable. They asked that whether the restricted-orientable partial Petrial polynomial of an arbitrary ribbon graph $G$ is even-interpolating. Later, Chen et al. \cite{Counterexample_ChenYichao} gave a negative answer.

The {\em join} $G_1\vee G_2$ of two disjoint ribbon graphs $G_1$ and $G_2$ is obtained by selecting an arc $p_i$ on the boundary of a vertex-disk in each $G_i$ for $i=1,2$ between two consecutive ribbon ends and identifying  $p_1$ and $p_2$ to merge the two vertex-disks into one \cite{Join_Moffatt}.~A ribbon graph $G$ is called {\em prime} if there do not exist non-empty subgraphs $G_1,G_2$ such that $G=G_1\vee G_2$.~Gross, Mansour, and Tucker \cite{2021_II} showed that the partial Petrial polynomial respects the join operation.~Specifically,~$^{\partial}{\varepsilon^{\times}_{G_1\vee G_2}}(z)=^{\partial}{\varepsilon^{\times}_{G_1}}(z)\cdot^{\partial}{\varepsilon^{\times}_{G_2}}(z)$. 
Thus, the study of the partial Petrial polynomial for a ribbon graph reduces to its study for prime ribbon graphs.~
Furthermore, they derived explicit formulas and recursive relations for certain families of ribbon graphs, including ladder graphs.


A \textit{bouquet} is a ribbon graph with a single vertex.~
Yan and Li \cite{Petrial_Yan} showed that the partial Petrial polynomial of a bouquet is determined by its intersection graph.~This means that two bouquets with the same intersection graph have identical partial Petrial polynomial. They also computed this polynomial of the bouquets whose intersection graphs are either a complete graph or a path. 

\vspace{-0.3em}
\begin{thm}\emph{(\cite{Petrial_Yan})}\label{first order_K_n_Petrial}
	Let $B_n$ be a prime bouquet with $n$ edges, where $n \geq 2$. Then
	\vspace{-0.5em}
	\[^{\partial}\varepsilon_{B_n}^{\times}(z) =  a_1z + a_2z^2 + \cdots + a_nz^n,\]
	
	\vspace{-0.5em}
	\noindent where $a_i \neq 0$ for all $1 \leq i \leq n$, if and only if $I(B_n)$ is a complete graph.
\end{thm}

\vspace{-1em}
\begin{thm}\emph{(\cite{Petrial_Yan})}\label{K_n}
	Let $B_n$ be a prime bouquet  such that $I(B_n)=K_n$, where \(n \geq 2\). Then  
	\[^{\partial}\varepsilon_{B_n}^{\times}(z) = 
	\begin{cases} 
		z^n + \sum_{i=1}^n \binom{n}{n+1-i} z^i, & \text{if } n \text{ is even}, \\ 
		z^{n-1} + \sum_{i=1}^n \binom{n}{n+1-i} z^i, & \text{if } n \text{ is odd}. 
	\end{cases}\]
\end{thm}

\noindent If the intersection graph is a path, then the polynomial is binomial. 
\vspace{-0.8em}
\begin{thm}\emph{(\cite{Petrial_Yan})}\label{P_n}
	Let $B_n$ be a prime bouquet with $n$ edges such that $I(B_n)=P_n$, $n\geq 1$. Then 
	\vspace{-0.6em}
	\[^{\partial}\varepsilon_{B_n}^{\times}(z) = 
	\begin{cases} 
		\left( \frac{2^n-1}{3} \right) z^{n-1} + \left( \frac{2^{n+1}+1}{3} \right) z^n, & \text{if } n \text{ is even}, \\ 
		\left( \frac{2^n+1}{3} \right) z^{n-1} + \left( \frac{2^{n+1}-1}{3} \right) z^n, & \text{if } n \text{ is odd}. 
	\end{cases}\]
\end{thm}

\vspace{-0.3em}
\indent We extend the work of \cite{Petrial_Yan} by determining the partial Petrial polynomial for bouquets whose intersection graphs are cycles. We also provide a complete characterization of the prime bouquets for which the lowest degree among the nonzero coefficients in the partial Petrial polynomial is $2$. Furthermore, we explicitly compute $^{\partial}{\varepsilon_{B_n}^{\times}}(z)$ if the lowest-degree term of  $^{\partial}{\varepsilon_{B_n}^{\times}}(z)$ is of degree $2$.\\
\indent The rest of this paper is organized as follows.~Section~2 introduces the necessary notations and foundational properties used throughout the paper.~Section~3 develops a recursion formula for the partial Petrial polynomial of bouquets whose intersection graphs are cycles (Theorem \ref{thm:main1}). Section~4 provides a characterization of the bouquet $B_n$ and presents exact explicit expression of the partial Petrial polynomial $^{\partial}{\varepsilon^{\times}_{B_n}}(z)$ if it contains a second-order term (Theorems \ref{thm:main3} and \ref{thm:main4}). As corollaries, we determine $^{\partial}{\varepsilon^{\times}_{B_n}}(z)$ when $I(B_n)$ is the complete bipartite and tripartite graph.

\section{Preliminaries}
In this section, 
the foundational definitions and some lemmas  
necessary for our main result are given. 
Unless stated otherwise, all graphs are assumed to be finite and undirected.
\subsection{Basic Definitions and Notation}
For a ribbon graph $G$, we denote its set of  faces and connected components by  $F(G)$ and $c(G)$, respectively. A ribbon graph is considered {\em empty} if its edge set is empty. The {\em Euler characteristic} of $G$ is given by $\chi(G)=|V(G)|-|E(G)|+|F(G)|$, and its {\em Euler genus} is $\varepsilon(G)=2c(G)-\chi(G)$. 

A loop of a bouquet is {\em twisted} if it is homeomorphic to a M\"obius band and {\em untwisted} if it is homeomorphic to an annulus.~The {\em signed rotation} of a bouquet is a cyclic ordering of its half-edges at the vertex-disk, with each half-edge assigned a sign, where the two half-edges of an untwisted loop have the same sign (the sign can be omitted), while those of a twisted loop have different signs. Two edges in a bouquet are {\em interlaced} if their half-edges alternate in the signed rotation. A loop is {\em trivial} if it is not interlaced with any other loop.
The orientable and nonorientable bouquets with Euler genus $2$ are called {\em toroidal bouquets} and {\em Klein bottle bouquets}, respectively.

A sign graph is a graph whose vertices are labeled with the sign $+$ or $-$.  The \textit{intersection graph} $I(B_n)$ of a bouquet $B_n$ has the loops of $B_n$ as its vertex set, and an edge exists between two vertices if the corresponding loops are interlaced. The \emph{signed intersection graph} $SI(B_n)$ is a variant of $I(B_n)$ where each vertex is assigned a sign corresponding to its loop type $+$ for untwisted, $-$ for twisted. 
The \emph{adjacency matrix} $\text{adj}(SI(B_n))$ of the signed intersection graph of the bouquet $B_n$ is obtained from the adjacency matrix $\text{adj}(I(B_n))$ of the intersection graph $I(B_n)$ by assigning $1$ to the diagonal entries corresponding to vertices of $SI(B_n)$ with sign $-$.\\
\indent Throughout this paper, let $\mathbb{Z}_n$ be the set $\{1,2,\cdots,n\}$. 
Let $\bm{O}_{s, t}$, $\bm{I}_{s, t}$, and $\bm{J}_{s, t}$ denote the $s \times t$ zero matrix, identity matrix, and all-ones matrix, respectively. Let $n,k$ be positive integers with $n \geq k$, and let $A \subseteq \mathbb{Z}_n$. We define $\bm{D}_{n,A}$ as the $n \times n$ diagonal matrix whose $(i,i)$-th entry is $1$ if $i \in A$ and $0$ otherwise. In particular, for $A = \{1,2,\ldots,k\}$, we write $\bm{D_{n,A}}$ as $\bm{D}_{n,k} := \bm{D}_{n,{\{1,2,\ldots,k\}}}$.\\
\indent For a polynomial $g(x)=\sum_{i=0}^{n}a_ix^{i}$, the {\em support} of the polynomial $g(x)$ is the set $supp(g)=\{i|a_i\neq 0\}$ of indices of the non-zero coefficients. Let 
$s=\min {supp(g)}$, $t=\max {supp(g)}$. 
The polynomial $g(x)$ is called interpolating if $supp(g)=\{j|s\leq j\leq t,j\in \mathbb{Z}\}$.
\subsection{Key Lemmas}

In this subsection, we provide some key lemmas for our arguments.

We now recall the well-known formula for calculating the Euler genus of a ribbon graph.
\vspace{-0.2em}
\begin{lem}\emph{(\cite{Yanqi_Intersection})}\label{varepsilon_ribbon}
Let $G$ be a ribbon graph. We have $\varepsilon(G)=2c(G)-|V(G)|+|E(G)|-|F(G)|$.
\end{lem}

\vspace{-0.2em}
For the specialized case of a bouquet, this formula simplifies considerably.
\vspace{-0.2em}
\begin{lem}\emph{(\cite{Yanqi_Intersection})}\label{varepsilon}
If $B_n$ is a bouquet, then $\varepsilon(B_n)=1+|E(B_n )|-|F(B_n)|$.
\end{lem}

\vspace{-0.2em}
The next lemma demonstrates a crucial connection between the partial Petrial polynomial and the intersection graph of a bouquet.
\vspace{-0.2em}
\begin{lem}\emph{\cite{Matroids_Yan}}\label{Petrial_orientable}
If bouquets $B_n^{\prime}$ and $B_n^{\prime\prime}$ have the same intersection graph, then $^{\partial}{\varepsilon^{\times}_{B_n^{\prime}}}(z)=^{\partial}{\varepsilon^{\times}_{B_n^{\prime\prime}}}(z)$.
\end{lem}

\vspace{-0.2em}
The following lemma shows that studying the partial Petrial polynomial of a ribbon graph is equivalent to studying the polynomial of any of its partial Petrial graphs.
\vspace{-0.2em}
\begin{lem}\emph{\cite{Petrial_Yan}}\label{G,GA}
Let $G$ be a ribbon graph and $A\subseteq E(G)$. Then  $^{\partial}{\varepsilon^{\times}_{G}}(z)=^{\partial}{\varepsilon^{\times}_{G^{\times|A}}}(z)$.
\end{lem}



\vspace{-0.2em}
A key property of the partial Petrial polynomial is that it is interpolating, meaning its coefficients are non-zero for all integers between the minimum and maximum degrees.
\vspace{-0.2em}
\begin{lem}\emph{(\cite{2021_II})}\label{interpolation of Petrial} 
For any ribbon graph $G$, the partial Petrial polynomial $^{\partial}{\varepsilon^{\times}_{G}}(z)$ is interpolating.
\end{lem}

\vspace{-0.2em}
We also rely on the result for the highest degree of the partial Petrial polynomial.
\vspace{-0.2em}
\begin{lem}\emph{(\cite{2021_II})}\label{highest-term of Petrial} 
If $G$ is connected, 
then the highest degree of $^{\partial}{\varepsilon^{\times}_{G}}(z)$ is $|E(G)|-|V(G)| + 1$.
\end{lem}

\vspace{-0.2em}
The rank of a matrix $A$ over $\mathbb{GF}(2)$ is denoted $\text{rank}(A)_{\mathbb{Z}_{2}}$. Finally, we use a known result that relates the rank of this matrix to the number of edges and faces of the bouquet.
\vspace{-0.2em}
\begin{lem}\emph{\cite{weight}}\label{rank=1+e-f}
For a bouquet $B_n$, then $\text{rank}(\text{adj}(SI(B_n)))_{\mathbb{Z}_2}=|E(B_n)|-|F(B_n)|+1$.
\end{lem}

\vspace{-0.8em}
\begin{lem}\label{rankofK_n}
Let $A=\bm{J}_{n,n}-\bm{I}_{n,n}$ be the $n\times n$ matrix with $n\geq 2$, then 
	\vspace{-0.3em}
	\[\text{rank}{(A)}_{\mathbb{Z}_2}=
	\begin{cases} 
		n, & \text{if } n \text{ is even}, \\ 
		n-1, & \text{if } n \text{ is odd}.
	\end{cases}\]
\end{lem}
\begin{proof}
	The determinant of $A$ over the integers is given by $\det(A)=(-1)^n(n-1)$.~Thus one has that the determinant of $A$ over $\mathbb{Z}_2$ is $0$ if $n$ is odd and $1$ otherwise. The result holds.
\end{proof}



\section{Partial Petrial Polynomial for Cycles}

In this section, we calculate the partial Petrial polynomial of bouquets whose intersection graphs are cycles. First, we introduce two key lemmas.
\vspace{-0.6em}
\begin{lem}\label{(1+z)*}
Let $G_n$ be a ribbon graph with $n$ edges. If the ribbon graph $G_{n+1}$ is obtained from $G_n$ by embedding a new edge $e$ whose two endpoints lie in the same face of $G_n^{\times |A}$ for any $A\subseteq E(G_n)$, then the partial Petrial polynomial of $G_n$ satisfies the recurrence relation
\vspace{-0.8em}
\[^{\partial}{\varepsilon_{G_{n+1}}^{\times}}(z)=(1+z)^{\partial}{\varepsilon_{G_n}^{\times}}(z).\]
\end{lem}
\vspace{-1.5em}
\begin{proof}
For any $A \subseteq E(G_n)$, 
since the two endpoints of $e$ are always contained in a single face of $G_n^{\times|A}$, embedding the edge $e$ in $G_n$ yields either
\vspace{-0.3em}
\[
\begin{cases}
	|F(G_{n+1}^{\times|(A \cup \{e\})})| = |F(G_n^{\times|A})|+1, \\
	|F(G_{n+1}^{\times|A})| = |F(G_n^{\times|A})|,
\end{cases}
\quad \text{or} \quad\quad
\begin{cases}
	|F(G_{n+1}^{\times|(A \cup \{e\})})| = |F(G_n^{\times|A})|, \\
	|F(G_{n+1}^{\times|A})| = |F(G_n^{\times|A})|+1,
\end{cases}
\]

\vspace{-0.3em}
\noindent which implies that either
\vspace{-0.3em}
\[
\begin{cases}
\varepsilon(G_{n+1}^{\times|(A \cup \{e\})}) = \varepsilon(G_n^{\times|A}), \\
\varepsilon(G_{n+1}^{\times|A}) = \varepsilon(G_n^{\times|A})+1,
\end{cases}
\quad \text{or} \quad\quad
\begin{cases}
\varepsilon(G_{n+1}^{\times|(A \cup \{e\})}) = \varepsilon(G_n^{\times|A})+1, \\
\varepsilon(G_{n+1}^{\times|A}) = \varepsilon(G_n^{\times|A}),
\end{cases}
\]

\vspace{-0.3em}
\noindent by Lemma \ref{varepsilon_ribbon}. Thus,
\vspace{-0.3em}
\[^\partial\varepsilon^{\times}_{G_{n+1}}(z)=\sum_{A \subseteq E(G_n)}(z^{\varepsilon(G_{n+1}^{\times|(A\cup \{e\})})} + z^{\varepsilon(G_{n+1} ^{\times|A})})=
\sum_{A \subseteq E(G_n)}(z^{\varepsilon(G_n^{\times|A})} + z^{\varepsilon(G_n ^{\times|A})+1})
= \sum_{A \subseteq E(G_n)} (1+z)~ z^{\varepsilon(G_n^{\times|A})}.\]

\vspace{-0.3em}
\noindent Hence,
\vspace{-0.3em}
\[
^\partial \varepsilon^{\times}_{G_{n+1}}(z) \;=\; (1+z)\,^\partial \varepsilon^{\times}_{G_n}(z),
\]

\vspace{-0.6em}
\noindent as required.
\end{proof}

\vspace{-1.2em}
\begin{lem}\emph{\cite{Petrial_Yan}}\label{faceofpath}
If $B_n$ is a bouquet with $n$ edges such that $I(B_n)=P_n$, then $|F(B^{\times|A})|=1$ or $2$ for any $A\subseteq E(B_n)$.
\end{lem}

\vspace{-0.5em}
The main result in this section is stated as follows.
\vspace{-0.6em}
\begin{thm}\label{thm:main1}
Let $B_n$ be a bouquet with $n$ edges. If  
$I(B_n)=C_n$ with $n\geq 3$, then 
\vspace{-0.3em}
\[^{\partial}{\varepsilon_{B_n}^{\times}}(z) = 
\begin{cases} 
\frac{2^{n-1}-1}{3} z^{n-2} +2^{n-1}z^{n-1} + \frac{2^{n}+1}{3}  z^n, & \text{if } n \text{ is odd}, \\ 
\frac{2^{n-1}+1}{3}z^{n-2} + 2^{n-1}z^{n-1} + \frac{2^{n}-1}{3} z^n, & \text{if } n \text{ is even}. 
\end{cases}\]
\end{thm}
\begin{proof}
Let $e\in E(B_n)$. 
Since $I(B_n)=C_n$, the intersection graph of $B_n\setminus e$ is a path (indeed $I(B_n\backslash e)\cong P_{n-1}$). By Lemma~\ref{faceofpath}, one has that $|F((B_n\backslash e)^{\times|A})|=1$ or $2$ for every $A\subseteq E(B_n\backslash e)$.\\
\indent We claim that, for every $A\subseteq E(B_n\backslash e)$, the two endpoints of $e$ lie in the same face of $(B_n\backslash e)^{\times|A}$. For contradiction, suppose that for some $A$, the two endpoints of $e$ lie in distinct faces of $(B_n\backslash e)^{\times|A}$. Label the three relevant boundary arcs as shown in Fig.~\ref{Fig_C_n} (red, pink, blue), so that the two ends of $e$ lie on the red and pink arcs. If the red and pink arcs are in different faces, then one of them, say the red arc, shares a face with the blue arc. Insert an edge $e'$ whose endpoints lie on the blue and red arcs. Then $I((B_n\backslash e)\cup\{e'\})\cong P_n$, and after taking the appropriate partial Petrial $B_n^{\times |A}$ or $B_n^{\times |A\cup\{e\}}$, we obtain a ribbon graph with three faces, contradicting that the number of faces in every partial Petrial of a path-type bouquet is either $1$ or $2$. Hence the two endpoints of $e$ always lie in the same face of $(B_n\backslash e)^{\times|A}$.
Therefore, the hypothesis of Lemma \ref{(1+z)*} applies, and
\vspace{-0.6em}
\[^\partial\varepsilon^{\times}_{B_n}(z)=(1+z)\,^\partial\varepsilon^{\times}_{B_n\backslash e}(z).\]

\vspace{-0.6em}
\noindent Finally, $I(B_n\backslash e)\cong P_{n-1}$, so $^{\partial}\varepsilon^{\times}_{B_n\backslash e}(z)$ is 
given in Theorem \ref{P_n}. Multiplying this polynomial by $(1+z)$ yields the claimed closed form for $^{\partial}\varepsilon^{\times}_{B_n}(z)$.
\begin{figure}[htbp]
	\centering
	\includegraphics[
	width=0.13\textwidth]{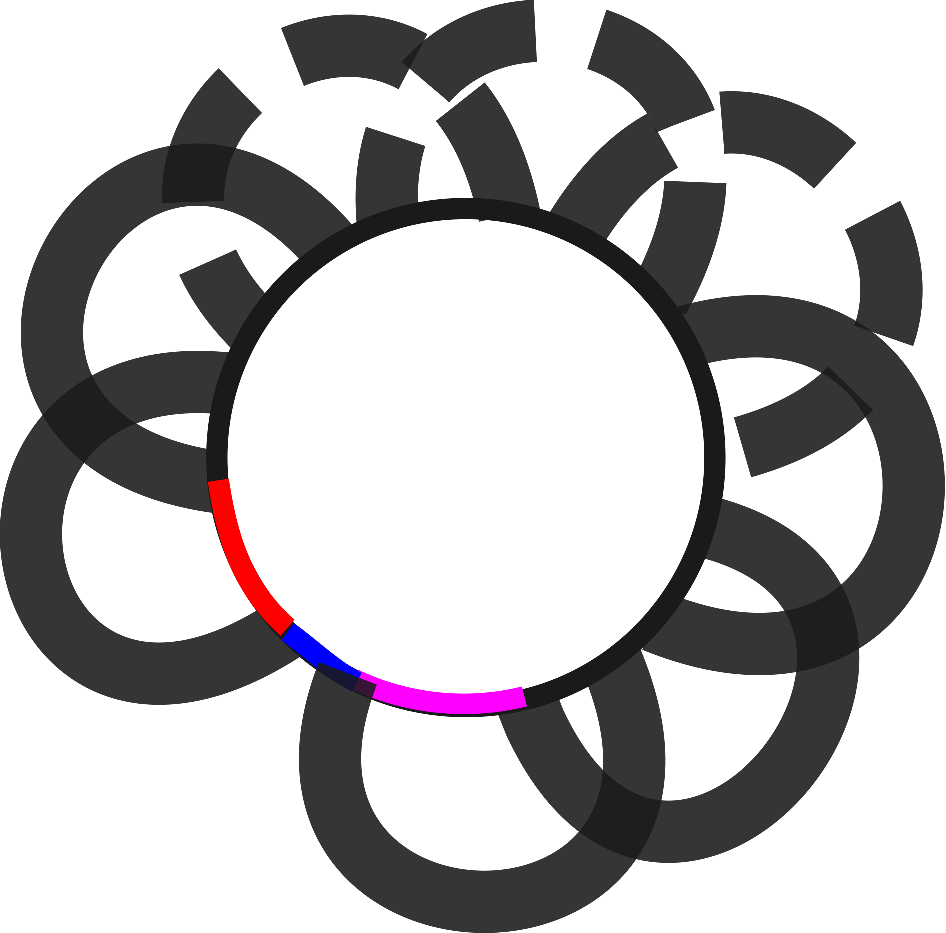}\\
	\caption{The bouquet $B_{n}\backslash e$ with $I(B_{n}\backslash e)=P_{n-1}$.}
	\label{Fig_C_n}
\end{figure}
\end{proof}

\section{Partial Petrial Polynomials of Bouquets with Lowest Degree 2}

Yan and Li gave a characterization of the bouquet $B_n$ such that $^{\partial}{\varepsilon^{\times}_{B_n}}(z)$ has a nonzero 
first-order term. 
In this section, we characterize the  bouquet $B_n$ for which $^{\partial}{\varepsilon^{\times}_{B_n}}(z)$ has a lowest nonzero term of degree 2. 

\subsection{Characterization of Bouquets}

By Theorems \ref{first order_K_n_Petrial} and \ref{K_n}, $^{\partial}{\varepsilon_{B_n}^{\times}}(z)$ has both 
first-order and  
second-order terms if~and~only~if~$n\geq 2$ and $I(B_n)=K_n$.~Thus, to characterize the bouquet $B_n$ for which $^{\partial}{\varepsilon_{B_n}^{\times}}(z)$ contains~the~second-order term, we only need to consider that the lowest-degree term of $^{\partial}{\varepsilon_{B_n}^{\times}}(z)$ is of degree $2$.

\vspace{-0.3em}
\begin{lem}\emph{(\cite{Yangyan})}\label{formofsurfaces}
Let $B_n$ be a bouquet with $n$ edges and $n\geq 2$. If $\varepsilon(B_n)=2$, then the signed rotation of $B_n$ has the following form:
\vspace{-0.3em}
\begin{itemize}
\item [\rm (1)] $B_n$ is a prime toroidal bouquet if and only if the signed rotation of $B_n$ has the form
\vspace{-0.5em}
\begin{equation}\tag {2.1}
a_1 \cdots a_ib_1\cdots b_jd_1\cdots d_ka_i\cdots a_1b_j \cdots b_1d_k\cdots d_1,
\end{equation}

\vspace{-0.8em}
\noindent where $i, j \geq 1, k \geq 0$ and $i + j + k = n$.
\vspace{-0.3em}
\item [\rm (2)] $B_n$ is a prime Klein bottle bouquet if and only if the signed rotation of $B_n$ has the form
	\vspace{-0.5em}
	\begin{equation}\tag{2.2}
		a_1\cdots a_id_1\cdots d_k(-a_1)\cdots (-a_i)b_1 \cdots b_jd_k\cdots d_1(-b_1)\cdots (-b_j),
	\end{equation}
	
\vspace{-1em}
\noindent where $n=i+j+k$, $i,k\geq 1$ and $j\geq0$.
\end{itemize}
\end{lem}

The following lemma can be directly derived from Lemma \ref{formofsurfaces} and the definition of  the intersection graph. 
\vspace{-0.3em}
\begin{lem}\label{lem_SI_B}
For a prime bouquet $B_n$, the following properties hold:
\vspace{-0.3em}
\begin{itemize}
\item [\rm (1)]  If $B_n$ is a toroidal bouquet, then 
$I{(B_n)}=K_{i,j}$ or $K_{i,j,k}$, where $n=i+j$ or $n=i+j+k$;
\vspace{-0.3em}
\item [\rm (2)] If $B_n$ is a Klein bottle bouquet, then $I{(B_n)}=F_{i,j,k}$, where $n=i+j+k$, $i,k\geq 1,j\geq0$. 
\end{itemize}
\end{lem}

Let $K_{i,j,k}$ denote the complete tripartite graph when $k\neq0$, and let $K_{i,j,0}$ denote the complete bipartite graph $K_{i,j}$. Define $F_{i,j,k}$  as follows.
\vspace{-0.3em}
\begin{defi}\label{def:F}
A graph $G$ is called $F_{i,j,k}$ if its vertex set $V(G)$ can be partitioned into three disjoint subsets $V_1$, $V_2$, and $V_3$, i.e., $V(G)=V_1\cup V_2\cup V_3$, satisfying the following two conditions:
\vspace{-1em}
\noindent
\begin{itemize}
	\vspace{-0.8em}
    \item [\rm (1)] The induced subgraphs $G[V_1]$ and $G[V_2]$ are complete graphs, $G[V_1] \cong K_i$ and $G[V_2] \cong K_j$, and there are no edges between $V_1$ and $V_2$ in $G$. It is required that $i\geq 1$ and $j\geq 0$.
    \vspace{-0.6em}
    \noindent
    \item [\rm (2)] The induced subgraph $G[V_3]$ is a stable set (an independent set) with $|V_3|=k\geq 1$. Every vertex $v\in V_3$ is adjacent to all vertices in $V_1\cup V_2$.
\end{itemize}
\end{defi}

\begin{thm}\label{thm:main2}
For~a~prime~bouquet~$B_n$~with~$n\geq3$~edges,~the lowest-degree~term~of~the~partial~Petrial~ polynomial~$^{\partial}{\varepsilon_{B_n}^{\times}}(z)$~is~of~degree~$2$~if~and~only~if~$I(B_n)$~satisfies~one~of~the~following~conditions
\vspace{-1.8em}
\begin{itemize}
	\item [\rm (1)] $I(B_n)=K_{i,j,k}$, where $k=0$ \text{or} $(i+j \geq 3 \text{ and } k \neq 0)$;
	\vspace{-0.3em}
	\item [\rm (2)] $I(B_n)=F_{i,j,k}$, where $j \neq 0$ \text{or} $(j=0 \text{ and } k \geq 2)$.
\end{itemize}
\end{thm}
\begin{proof}
$(\Leftarrow)$ Suppose $I(B_n)$ is either $K_{i,j,k}$ with $k=0$ or $i+j\geq 3, k\neq 0$, or $F_{i,j,k}$ with $j \neq 0$ or $j=0, k\geq 2$, for $i+j+k=n\geq 3$.
In these cases, there exists a subset $A$ of edges such that the partial Petrial $B_n^{\times|A}$ is either a toroidal bouquet (Euler genus $2$) or a Klein bottle bouquet (Euler genus $2$) as shown in Lemma \ref{formofsurfaces} and Lemma \ref{lem_SI_B}.
By Theorem \ref{first order_K_n_Petrial}, the lowest nonzero term occurs at degree $2$ in the partial Petrial polynomial.\\
($\Rightarrow$) Suppose the lowest-degree term of $^{\partial}{\varepsilon_{B_n}^{\times}}(z)$ is degree $2$. Then, for some subset~$A\subseteq E(B_n)$, $B_n^{\times|A}$ has Euler genus $2$, and by the classification of allowable intersection graphs (Lemma \ref{lem_SI_B}), this occurs only when ~$I(B_n)$~ is ~$K_{i,j,k}$~ or ~$F_{i,j,k}$~ within the specified parameter ranges. If $I(B_n)$ were $K_n$, all degrees from $1$ to $n$ would appear by Theorem \ref{first order_K_n_Petrial}, which is excluded by assumption.
%
\end{proof}

\subsection{Computation of Partial Petrial Polynomials}
To complete the partial Petrial polynomial for bouquets contains a second-order term, by Theorem \ref{thm:main2}, we only need to compute these polynomials for the bouquets  whose intersection graphs are $K_{i,j,k}$ and $F_{i,j,k}$, respectively.  We first introduce an equivalent representation of the partial Petrial polynomial of bouquets from the perspective of matrix rank.  
\begin{lem}\label{rankofpoly_Petrial}
	For a prime bouquet $B_n$, then the partial Petrial polynomial of $B_n$ is 
	as follows
	\vspace{-0.3em}
	\[^{\partial}{\varepsilon_{B_n}^{\times}}(z)=\sum_{A \subseteq E(B_{n})} z^{\text{rank}(\text{adj}(SI(B_n^{\times|A})))_{\mathbb{Z}_2}}=\sum_{A \subseteq \mathbb{Z}_{n}} z^{\text{rank}(\text{adj}(I(B_n))+\bm{D}_{n,A})_{\mathbb{Z}_2}}.\]
\end{lem}
\vspace{-1em}
\begin{proof}
	The partial Petrial polynomial is defined as $^{\partial}\varepsilon_{B_n}^{\times}(z) = \sum_{A \subseteq E(B_n)} z^{\varepsilon(B_n^{\times|A})}$. By Lemma \ref{G,GA}, studying the partial Petrial polynomial of a bouquet is equivalent to studying the polynomial of any of its partial Petrial graphs.  
	Assuming that the base bouquet $B_n$ is orientable (i.e., all loops are untwisted). Our goal is to find an equivalent expression for $\varepsilon(B_n^{\times|A})$ for any $A \subseteq E(B_n)$.\\
	\indent An arbitrary partial Petrial $B_n^{\times|A}$ is itself a bouquet. Its Euler genus is given by $\varepsilon(B_n^{\times|A}) = 1 + |E(B_n)| - |F(B_n^{\times|A})|$ by Lemma \ref{varepsilon}, and the rank of the adjacency matrix of its signed~intersection graph is $\text{rank}(\text{adj}(SI(B_n^{\times|A})))_{\mathbb{Z}_2} = |E(B_n)| - |F(B_n^{\times|A})| + 1$ by Lemma \ref{rank=1+e-f}. By directly comparing these two formulas, we can equate the Euler genus with the rank of the matrix, that is
	\vspace{-0.2em}
	\[\varepsilon(B_n^{\times|A}) = \text{rank}(\text{adj}(SI(B_n^{\times|A})))_{\mathbb{Z}_2}.\]
	
	\vspace{-0.2em}
	\indent The final step is to relate this matrix back to the intersection graph of the original bouquet.~
	The partial Petrial  $B_n^{\times|A}$ twists the edges in set $A$.  
	Since $B_n$ is orientable, 
	the matrix $\text{adj}(SI(B_n^{\times|A}))$ is simply the adjacency matrix of the base intersection graph, $\text{adj}(I(B_n))$, plus the diagonal matrix $\bm{D}_{n,A}$,  by the definition of adjacency matrix of signed intersection graph. Therefore, we have:
	\vspace{-0.3em}
	$$\varepsilon(B_n^{\times|A}) = \text{rank}(\text{adj}(I(B_n)) + \bm{D}_{n,A})_{\mathbb{Z}_2}.$$
	
	\vspace{-0.3em}
	\noindent 	Substituting this expression 
	for the exponent back into the definition of the polynomial completes the proof.
\end{proof}

\vspace{-0.3em}
Subsequently, we introduce the following notational conventions for our computation.

(1) 
Let $\sum_{a=c_1}^{c_2} \binom{b}{a}=0$ if 
$c_1 > c_2$;
 
(2) Let $\binom{j}{k} = 0$ if $k<0$ or $k > j$ or $j=0$, and define
\vspace{-0.1em}
	\begin{align*}
		\overline{\binom{j}{k}}= 
		\begin{cases}
			\binom{j}{k},& \text{if~} 0\leq k\leq j-1,\\
			0,	& \text{if~} k<0 \text{ or } k\geq j.
		\end{cases}
	\end{align*}

\begin{thm}\label{thm:main3}
Let $B_n$ be a bouquet with $n \geq 3$ edges, and suppose its intersection graph is $K_{i,j,k}$, where either $k = 0$ or $i + j \geq 3$ with $k \neq 0$.  
Then the partial Petrial polynomial is given by
\begin{align*}
	^{\partial}{\varepsilon_{B_n}^{\times}}(z)= 
	\begin{cases}
		z^{n}+\sum\limits_{d=2}^{n}g_dz^d,
		& \text{if~} n_o=0,\\
		z^{n+1-n_o}+\sum\limits_{d=2}^{n}g_dz^d,
		& \text{if~} n_{o}\geq1,
	\end{cases}
\end{align*}
\noindent where each coefficient $g_d$ is expressed explicitly in terms of $i, j, k$ as
\[\small{
	\begin{aligned}
		g_d =\;
		&\overline{\binom{i}{d-2\left\lceil \frac{j + k}{2} \right\rceil}} +
		\overline{\binom{j}{d - 2\left\lceil \frac{i + k}{2} \right\rceil} }+ 
		\overline{\binom{k}{d - 2\left\lceil \frac{i + j}{2} \right\rceil}}
		+\sum\limits_{a=0}^{i-1}
		\binom{i}{a}\overline{\binom{j}{d - 2 - k - a}}\\
		&+ \sum\limits_{a=0}^{i-1}\binom{i}{a}
		\overline{\binom{k}{d - 2 - j - a}} + 
		\sum\limits_{a = 0}^{i} \sum\limits_{b=0}^{j-1}
		\binom{i}{a} \binom{j}{b}~\overline{\binom{k}{d - 2 - a - b}}
\end{aligned}}
\]
and $n_o$ is the number of odd elements among $i, j$ and $k$.
\end{thm}
\begin{proof}
To determine the partial Petrial polynomial of $B_n$, we first evaluate the rank of the matrix  $\text{adj}(I(B_n))+\bm{D}_{n,X}$ over $\mathbb{GF}(2)$ for any $X\subseteq \mathbb{Z}_n$ by Lemma \ref{rankofpoly_Petrial}. Since $I{(B_n)}=K_{i,j,k}$, the vertex set is partitioned into subsets  $A=\{u_1,u_1,\cdots,u_i\}$,  $B=\{v_1,v_2,\cdots,v_j\}$ and $C=\{w_1,w_2,\cdots,w_k\}$. The adjacency matrix $\text{adj}(I(B_n))$ depends on the order of vertices with its first $i$, next $j$, and final $k$ rows/columns corresponding to $A$, $B$, and $C$, respectively, hence $\text{adj}(I(B_n))$ is given by
\vspace{-0.3em}
\[\text{adj}(I(B_n)) = 
\begin{bmatrix}
	\bm{O}_{i,i} & \bm{J}_{i,j} & \bm{J}_{i,k} \\
	\bm{J}_{j,i} & \bm{O}_{j,j} & \bm{J}_{j,k} \\
	\bm{J}_{k,i} & \bm{J}_{k,j} & \bm{O}_{k,k} \\
\end{bmatrix}.
\] 

\vspace{-0.3em}
\noindent We use $T_{\text{index}}$ to denote the set of numbers of rows in the adjacency matrix that correspond to the vertices in $T$ for any $T\subseteq A\cup B\cup C$. For any $X\subseteq \mathbb{Z}_n$, we divide $X$ into four cases based on whether $X$ fully contains exactly three, two, one, or none of the sets $A_{\text{index}}$, $B_{\text{index}}$, and $C_{\text{index}}$. 

We set $\overline{a}=|X\cap A_{\text{index}}|,\overline{b}=|X\cap B_{\text{index}}|$ and $\overline{c}=|X\cap C_{\text{index}}|$. 
Let $\mathbb{X}$ be the family of some subsets of $\mathbb{Z}_n$ and let  $P_{\mathbb{X}}=\sum\limits_{X\in\mathbb{X}} z^{{\text{rank}(\text{adj}(I(B_n))+\bm{D}_{n,X})}_{\mathbb{Z}_2}}$.

\noindent\textbf{Case 1: $\bm{X}$ contains all of $\bm{A_{\text{index}},B_{\text{index}}}$ and $\bm {C_{\text{index}}}$.}

In this case, 
one has $X=(A\cup B\cup C)_{\text{index}}$. Regardless of whether there is an odd number among $i, j, k$ (if so, without loss of generality, assume $i$ is odd), one has
\vspace{-0.3em}
\[\text{adj}(I(B_n))+\bm{D}_{n,X}=
\begin{bmatrix}
	\bm{I}_{i,i} & \bm{J}_{i,j} & \bm{J}_{i,k} \\
	\bm{J}_{j,i} & \bm{I}_{j,j} & \bm{J}_{j,k} \\
	\bm{J}_{k,i} & \bm{J}_{k,j} & \bm{I}_{k,k} \\
\end{bmatrix}
\rightarrow
\begin{bmatrix}
	\bm{J}_{1,1} & \bm{J}_{1,i-1} & \bm{J}_{1,j} & \bm{J}_{1,k} \\
	\bm{J}_{i-1,1} & \bm{(J-I)}_{i-1,i-1} & \bm{O}_{i-1,j} & \bm{
		O}_{i-1,k} \\
	\bm{O}_{j,1}& \bm{O}_{j,i-1} & \bm{(J-I)}_{j,j} & \bm{O}_{j,k} \\
	\bm{O}_{k,1}& \bm{O}_{k,i-1} & \bm{O}_{k,j} & \bm{(J-I)}_{k,k} \\
\end{bmatrix},
\]

\vspace{-0.3em}
\noindent by some elementary transformations. 
\vspace{-0.3em}
According to Lemma \ref{rankofK_n}, one has
\vspace{-0.2em}
\begin{align*}
	\text{rank}(\text{adj}(I(B_n))+\bm{D}_{n,X})_{\mathbb{Z}_2}= 
	\begin{cases}
		{i+j+k}, & \text{if~} n_{o}=0,\\
		i+j+k+1-n_{o}, & \text{if~} n_{o}\geq1.
	\end{cases}  
\end{align*}

\vspace{-0.6em}
\noindent Let $\mathbb{X}_1=\{(A\cup B\cup C)_{\text{index}}\}$, we have 
\vspace{-0.6em}
\begin{align}\label{main3_eq1}
	P_{\mathbb{X}_1}=\begin{cases}
		z^{n}, & \text{if~} n_{o}=0,\\
		z^{n+1-n_{o}}, & \text{if~} n_{o}\geq1.
	\end{cases}  
\end{align}

\vspace{-0.3em}
\noindent\textbf{Case 2: $\bm{X}$ contains exactly two of $\bm{A_{\text{index}},B_{\text{index}}}$ and $\bm{C_{\text{index}}}$}.

Since $X$ fully contains exactly two of $A_{\text{index}},B_{\text{index}}$ and $C_{\text{index}}$, one has that $(A\cup B\cup C\backslash Y)_{\text{index}}\subseteq X\subsetneqq (A\cup B\cup C)_{\text{index}}$, where  $Y\in\{A,B,C\}$.~
Without loss of generality, consider $X$ such that~$(A\cup B)_{\text{index}} \subseteq X \subsetneq (A \cup B \cup C)_{\text{index}}$,~so $X$ misses at least one element of $C_{\text{index}}$.

We apply elementary transformations to the full rows and columns indexed by $C_{\text{index}}$ in matrix $\text{adj}(I(B_n))+\bm{D}_{n,X}$. Specifically, we perform simultaneous swaps between a row/column whose diagonal entry is $1$ and one whose diagonal entry is $0$. This is done until the first $|X\cap C_{\text{index}}|$ diagonal elements are $1$ and the remaining are $0$. 
Thus the modified adjacency matrix is then 
\vspace{-0.3em}
\[\text{adj}(I(B_n))+\bm{D}_{n,X}=
\begin{bmatrix}
	\bm{I}_{i,i} & \bm{J}_{i,j} & \bm{J}_{i,k} \\
	\bm{J}_{j,i} & \bm{I}_{j,j} & \bm{J}_{j,k} \\
	\bm{J}_{k,i} & \bm{J}_{k,j} & \bm{D}_{k,X\cap C_{\text{index}}} \\
\end{bmatrix}
\rightarrow 
 \begin{bmatrix}
	\bm{I}_{i,i} & \bm{J}_{i,j} & \bm{J}_{i,k} \\
	\bm{J}_{j,i} & \bm{I}_{j,j} & \bm{J}_{j,k} \\
	\bm{J}_{k,i} & \bm{J}_{k,j} & 
	\bm{D}_{k,\overline{c}} \\
\end{bmatrix}.\]

\vspace{-0.3em}
\noindent Regardless of whether there is an odd number among $i,j$ (if so, without loss of generality, assume $i$ is odd),  through some elementary transformations, we have that
\vspace{-0.6em}
\[\text{adj}(I(B_n))+\bm{D}_{n,X} \rightarrow 
\begin{bmatrix}
\bm{I}_{i,i} & \bm{O}_{i,j} & \bm{O}_{i,k}\\
\bm{O}_{j,i} & \bm{(J-I)}_{j,j} & \bm{O}_{j,k}\\
\bm{O}_{k,i} & \bm{O}_{k,j} & \bm{D}_{k,k^{\prime}} \\
\end{bmatrix},
\]

\vspace{-0.3em}
\noindent where $k^{\prime}=\overline{c}$ if $i,j$ are even; and $k^{\prime}=\overline{c}+1$, otherwise. By Lemma \ref{rankofK_n}, it follows that

\vspace{-0.8em}
\[\text{rank}(\text{adj}(I(B_n))+\bm{D}_{n,X})_{\mathbb{Z}_2}=2\left\lceil\frac{i+j}{2}\right\rceil +|X\cap C_{\text{index}}|.\]

\vspace{-0.3em}
\noindent Summing over all possible $X$ in this subclass, the generating function is
\vspace{-0.6em}
\[\sum_{(A\cup B)_{\text{index}} \subseteq X \subsetneqq (A\cup B\cup C)_{\text{index}}} z^{{\text{rank}(\text{adj}(I(B_n))+\bm{D}_{n,X})}_{\mathbb{Z}_2}}=\sum_{c=0}^{k-1}\binom{k}{c} z^{2\left\lceil\frac{i+j}{2}\right\rceil+c}.\] 

\vspace{-0.6em}
\noindent By symmetry, we have the analogous sums for the other two cases:
\vspace{-0.6em}
\begin{align*}
\sum_{a=0}^{i-1} \binom{i}{a} z^{2\left\lceil\frac{j+k}{2}\right\rceil + a}; \quad
\sum_{b=0}^{j-1} \binom{j}{b} z^{2\left\lceil\frac{i+k}{2}\right\rceil + b}; \quad
\sum_{c=0}^{k-1} \binom{k}{c} z^{2\left\lceil\frac{i+j}{2}\right\rceil + c}.
\end{align*}

\vspace{-0.6em}
\noindent 
Let $\mathbb{X}_2$ be the set of all subsets $X$ of $\mathbb{Z}_n$ that satisfy Case $2$. This yields the following sum: 
\vspace{-0.6em}
\begin{align*}
	P_{\mathbb{X}_2}=\sum_{a=0}^{i-1}\binom{i}{a} z^{2\left\lceil\frac{j+k}{2}\right\rceil +a}+\sum_{b=0}^{j-1} \binom{j}{b} z^{2\left\lceil\frac{i+k}{2}\right\rceil +b} +\sum_{c=0}^{k-1} \binom{k}{c}z^{2\left\lceil\frac{i+j}{2}\right\rceil +c}.
\end{align*}

\vspace{-0.6em}
\noindent After simplification, we can sum the above expression over different degrees as follows:
\vspace{-0.3em}
\begin{align}\label{main3_eq2}
P_{\mathbb{X}_2}=\sum\limits_{d = 2}^{n}\left[\overline{\binom{i}{d-2\left\lceil\frac{j + k}{2} \right\rceil}} + \overline{\binom{j}{d - 2\left\lceil \frac{i + k}{2}\right\rceil}} + \overline{\binom{k}{d - 2\left\lceil \frac{i + j}{2} \right\rceil}}\right]z^{d}.
\end{align}
\noindent\textbf{Case 3: $\bm{X}$ contains exactly one of $\bm{A_{\text{index}},B_{\text{index}}, C_{\text{index}}}$.}

In this case, $X$ fully contains exactly one of $A_{\text{index}}$, $B_{\text{index}}$, or $C_{\text{index}}$. More precisely, for $Y \in \{A_{\text{index}}, B_{\text{index}}, C_{\text{index}}\}$ and $Z$ the remaining two sets, $Y\subseteq X$ and $Z\not\subseteq X$
. Without loss of generality, consider $A_{\text{index}}\subseteq X$ but $B_{\text{index}}\not\subseteq X$ and $C_{\text{index}}\not\subseteq X$.
After some elementary transformations, we have 
\vspace{-0.6em}
\[\text{adj}(I(B_n))+\bm{D}_{n,X}
\rightarrow 
\begin{bmatrix}
\bm{I}_{i,i} & \bm{J}_{i,j} & \bm{J}_{i,k} \\
\bm{J}_{j,i} & \bm{D}_{j,\overline{b}} & \bm{J}_{j,k} \\
\bm{J}_{k,i} & \bm{J}_{k,j} & \bm{D}_{k,\overline{c}}
\end{bmatrix}
\rightarrow	\bm{D}_{n,i+\overline{b}+\overline{c}+2}.
\]

\vspace{-0.6em}
\noindent Thus the rank of this matrix over $\mathbb{Z}_2$ is
\vspace{-0.6em}
$$\text{rank}(\text{adj}(I(B_n))+\bm{D}_{n,X})_{\mathbb{Z}_2}=i+|X\cap B_{\text{index}}|+|X\cap C_{\text{index}}|+2.$$

\vspace{-0.5em}
\noindent By symmetry, the same formula holds when $X$ fully contains exactly one of $B_{\text{index}}$ or $C_{\text{index}}$. Let $\mathbb{X}_3$ be the set of all subsets $X$ of $\mathbb{Z}_n$ that satisfy Case $3$. Thus, this yields the following sum: 
\vspace{-0.6em}
\begin{equation}\label{main3_eq3}
	\begin{aligned}
		\hspace{-0.8em}P_{\mathbb{X}_3}&= \sum\limits_{a=0}^{i-1}\sum\limits_{b=0}^{j-1} \binom{i}{a}\binom{j}{b}z^{a+b+k+2} +  \sum_{a=0}\limits^{i-1}\sum\limits_{c=0}^{k-1} \binom{i}{a}\binom{k}{c} z^{a+j+c+2} + \sum\limits_{b=0}^{j-1}\sum\limits_{c=0}^{k-1}\binom{j}{b}\binom{k}{c}z^{i+b+c+2}\\
		&=\sum\limits_{d=2}^{n}\left[\sum\limits_{a=0}^{i-1}
		\binom{i}{a}\left(\overline{\binom{j}{d - 2 - k - a}} +\overline{\binom{k}{d - 2 - j - a}}\right)+
		\sum\limits_{b=0}^{j-1}\binom{j}{b}\overline{\binom{k}{d - 2 - i - b}}\right]z^{d}.
	\end{aligned}
\end{equation}

\noindent\textbf{Case 4: $\bm{X}$ contains none of $\bm{A_{\text{index}},B_{\text{index}}, C_{\text{index}}}$.}

In this case, $X$ does not fully contain any of the sets $A_{\text{index}}$, $B_{\text{index}}$, and $C_{\text{index}}$; that is, $A_{\text{index}} \not\subseteq X$, $B_{\text{index}}\not\subseteq X$, and $C_{\text{index}}\not\subseteq X$. 
After some elementary transformations, we have that
\vspace{-0.3em}
\[\text{adj}(I(B_n))+\bm{D}_{n,X}\rightarrow
\begin{bmatrix}
	\bm{D}_{i,\overline{a}} & \bm{J}_{i,j} & \bm{J}_{i,k} \\
	\bm{J}_{j,i} & \bm{D}_{j,\overline{b}} & \bm{J}_{j,k} \\
	\bm{J}_{k,i} & \bm{J}_{k,j} & \bm{D}_{k,\overline{c}} \\
\end{bmatrix}
\rightarrow
\bm{D}_{n,\overline{a}+\overline{b}+\overline{c}+2} .
\]

\vspace{-0.6em}
\noindent Thus we obtain that $\text{rank}(\text{adj}(I(B_n))+\bm{D}_{n,X})_{\mathbb{Z}_2}=|X\cap A_{\text{index}}|+|X\cap B_{\text{index}}|+|X\cap C_{\text{index}}|+2.$
%
Let $\mathbb{X}_4$ be the set of $X$ that satisfies Case $4$, one has 
\vspace{-0.6em}
\begin{align}\label{main3_eq4}
P_{\mathbb{X}_4}=\sum_{a=0}^{i-1}\sum_{b=0}^{j-1}\sum_{c=0}^{k-1} \binom{i}{a}\binom{j}{b} \binom{k}{c} z^{a+b+c+2}=
\sum\limits_{d = 2}^{n}\sum\limits_{a = 0}^{i-1}\sum\limits_{b=0}^{j-1}
\binom{i}{a}\binom{j}{b}\overline{\binom{k}{d - 2 - a - b}}z^d.
\end{align}

\vspace{-0.3em}
Summarizing formulates (\ref{main3_eq1})-(\ref{main3_eq4}), the partial Petrial polynomial is
\vspace{-0.8em} $$^{\partial}{\varepsilon_{B_n}^{\times}}(z)=\sum\limits_{X\in\mathbb{Z}_n} z^{{\text{rank}(\text{adj}(I(B_n))+\bm{D}_{n,X})}_{\mathbb{Z}_2}}=\sum\limits_{l=1}^{4}\sum\limits_{X\in \mathbb{X}_l} z^{{\text{rank}(\text{adj}(I(B_n))+\bm{D}_{n,X})}_{\mathbb{Z}_2}}= P_{\mathbb{X}_1}+P_{\mathbb{X}_2}+P_{\mathbb{X}_3}+P_{\mathbb{X}_4}.$$

\vspace{-0.6em}
\noindent Sum up the coefficients of $z$ to the power of $d$, $2\leq d\leq n$. This completes the proof.
\end{proof}
\vspace{-1em}
\begin{coro}
Let $B_{i+1}$ be a bouquet with $i+1$ edges such that $I(B_{i+1})=K_{i,1}$ and $i\geq 2$. 
Then 
\vspace{-0.6em} \[^{\partial}{\varepsilon_{B_n}^{\times}}(z)= z^{i}+z^{i+1}+2z^2\left[(1+z)^{i}-z^i\right].\]
\end{coro}
\vspace{-1em}
\begin{coro}
	Let $B_{2i}$ be a bouquet with $2i$ edges such that $I(B_{2i})=K_{i,i}$ and $i\geq 2$. 
	Then 
	\vspace{-0.3em} \begin{align*}
		^{\partial}{\varepsilon_{B_n}^{\times}}(z)= 
		\begin{cases}
			z^{2i-1}+2\left[(1+z)^{i}-z^i\right]z^{i+1} + z^2\left[(1+z)^{i}-z^i\right]^2, 
			& \text{if~} i \text{ is odd};\\
			z^{2i}+2\left[(1+z)^{i}-z^i\right]z^{i} + z^2\left[(1+z)^{i}-z^i\right]^2, & \text{else}.\\
		\end{cases}  
	\end{align*}
\end{coro}
\vspace{-1em}
\begin{coro}
	Let $B_{3i}$ be a bouquet with $3i$ edges such that $I(B_{3i})=K_{i,i,i}$ and $i\geq 2$. 
	Then 
	\vspace{-0.3em} 
	\begin{align*}
		^{\partial}{\varepsilon_{B_n}^{\times}}(z)= 
		\begin{cases}
			z^{3i-2}+3z^{2i}\left[(1+z)^{i}-z^i\right] + 3z^{i+2}\left[(1+z)^{i}-z^i\right]^2 +
			z^{2}\left[(1+z)^{i}-z^i\right]^3, 
			& \text{if~} i\text{ is odd};\\
			z^{3i}+3z^{2i}\left[(1+z)^{i}-z^i\right] + 3z^{i+2}\left[(1+z)^{i}-z^i\right]^2 +
			z^{2}\left[(1+z)^{i}-z^i\right]^3, & \text{else}.\\
		\end{cases}  
	\end{align*}
\end{coro}

Next, we provide a definition that is used in the following theorem. 
\vspace{-0.6em}
\begin{defi}\label{def:N}
For an ordered triple $(i, j, k)$, define $N_{(i, j, k)}:=n_1 + n_2 + n_3$, where
\vspace{-0.3em}
\[
n_1 =
\begin{cases}
    1, & \text{if } i \text{ is odd}, \\
    0, & \text{if } i \text{ is even},
\end{cases}
\qquad
n_2 =
\begin{cases}
    1, & \text{if } j \text{ is odd}, \\
    0, & \text{if } j \text{ is even},
\end{cases}
\qquad
n_3 =
\begin{cases}
    0, & \text{if } k \text{ is odd}, \\
    1, & \text{if } k \text{ is even}.
\end{cases}
\]
\end{defi}
\begin{thm}\label{thm:main4}
Let $B_n$ be a prime bouquet with $n \geq 3$ edges whose intersection graph is $F_{i, j, k}$, where either ~$j \neq 0$, ~or ~$j = 0$~ and ~$k \ge 2$. Then the partial Petrial polynomial of $B_n$ is
\begin{align*}
	^{\partial}{\varepsilon_{B_n}^{\times}}(z)= 
	\begin{cases}
		z^{n+1-N}+\sum\limits_{d=2}^{n}g_dz^d,
		& \text{if~} N\geq 1, \\ 
		z^{n}+\sum\limits_{d=2}^{n}g_dz^d,
		& \text{if~} N=0,
	\end{cases}  
\end{align*}
\noindent where each coefficient $g_d$ is expressed explicitly in terms of $i, j, k$ as
\[\small{
	\begin{aligned}
		g_d =\;
		&\overline{\binom{i}{d-1-{2\left\lfloor \frac{j+k}{2} \right\rfloor}}} +
		\overline{\binom{j}{d-1-2\left\lfloor \frac{i+k}{2} \right\rfloor}}+
		\overline{\binom{k}{d - 2\left\lceil \frac{i + j}{2} \right\rceil}}+\sum\limits_{a=0}^{i-1}
		\binom{i}{a}\overline{\binom{j}{d - 2 - k - a}}\\
		& + \sum\limits_{a=0}^{i-1}
		\binom{i}{a}
		\overline{\binom{k}{d - 2 - j - a}} + \sum\limits_{a = 0}^{i} \sum\limits_{b=0}^{j-1}
		\binom{i}{a} \binom{j}{b} \overline{\binom{k}{d - 2 - a - b}}.
\end{aligned}}
\]
and $N = N_{(i, j, k)}$ is the number as defined in Definition \emph{\ref{def:N}}. 
\end{thm}
\begin{proof}
To determine $^{\partial}{\varepsilon_{B_n}^{\times}}(z)$ of $B_n$,~we first calculate the rank of the matrix  $\text{adj}(I(B_n))+\bm{D}_{n,X}$ over $\mathbb{GF}(2)$ for any $X\subseteq \mathbb{Z}_n$ by Lemma \ref{rankofpoly_Petrial}. Since $I{(B_n)}=F_{i,j,k}$, let~$A=\{u_1,u_1,\cdots,u_i\}$,$B=\{v_1,v_2,\cdots,v_k\}$, $C=\{w_1,w_2,\cdots,w_k\}$ be the partition of $V(F_{i,j,k})$ as defined in Definition \ref{def:F}. The adjacency matrix $\text{adj}(I(B_n))$ depends on the order of vertices with its first $i$, next $j$, and final $k$ rows/columns corresponding to $A$, $B$, and $C$, respectively, hence $\text{adj}(I(B_n))$ is given by
\vspace{-0.3em}
\[\text{adj}(I(B_n)) = 
\begin{bmatrix}
\bm{(J-I)}_{i,i} & \bm{O}_{i,j} & \bm{J}_{i,k} \\
\bm{O}_{j,i} & \bm{(J-I)}_{j,j} & \bm{J}_{j,k}\\ 
\bm{J}_{k,i} & \bm{J}_{k,j} & \bm{O}_{k,k}
\end{bmatrix}.
\]

\vspace{-0.3em}
\noindent For any ~$T \subseteq A\cup B\cup C$, let ~$T_{\text{index}}$~ denote the set of numbers of rows in the adjacency matrix corresponding to the vertices in ~$T$.  We classify all possible subsets ~$X$~ into four cases according to how many of the following conditions ~(C1-C3)~ are satisfied: three, two, one, or none
\vspace{-0.6em}
\begin{itemize}
	\item [\rm (C1).] $X\cap A_{\text{index}}=\emptyset$; \qquad\quad \rm (C2). $X\cap B_{\text{index}}=\emptyset$;  \qquad\quad \rm (C3). $C_{\text{index}}\subseteq X$.
\end{itemize}	

\vspace{-0.6em}
Let $\overline{a}=i-|X\cap A_{\text{index}}|,\overline{b}=j-|X\cap B_{\text{index}}|$ and $\overline{c}=|X\cap C_{\text{index}}|$. 
Let $\mathbb{X}$ be the family of some subsets of $\mathbb{Z}_n$ and  let $P_{\mathbb{X}}=\sum\limits_{X\in\mathbb{X}} z^{{\text{rank}(\text{adj}(I(B_n))+\bm{D}_{n,X})}_{\mathbb{Z}_2}}$.

\vspace{0.2em}
\noindent\textbf{Case 1: $\bm{X}$ satisfies all three of C1, C2, and C3.}
	
Since the set $X$ satisfies C1, C2 and C3, it follows that $X=C_{\text{index}}$. Regardless of whether there is an odd number among $i,j$ (if so, without loss of generality, assume $i$ is odd), through some elementary transformations, we have 
\vspace{-0.3em}
\[\text{adj}(I(B_n))+\bm{D}_{n,X}= 
\begin{bmatrix}
	\bm{(J-I)}_{i,i} & \bm{O}_{i,j} & \bm{J}_{i,k} \\
	\bm{O}_{j,i} & \bm{(J-I)}_{j,j} & \bm{J}_{j,k}\\ 
	\bm{J}_{k,i} & \bm{J}_{k,j} & \bm{I}_{k,k}
\end{bmatrix}
\rightarrow
\begin{bmatrix}
	\bm{Y}_{i,i} & \bm{O}_{i,j} & \bm{O}_{i,k} \\
	\bm{O}_{j,i} & \bm{(J-I)}_{j,j} & \bm{O}_{j,k}\\ 
	\bm{O}_{k,i} & \bm{O}_{k,j} & \bm{I}_{k,k}
\end{bmatrix},
\]

\vspace{-0.3em}
\noindent where $\bm{Y}=\bm{I}$ if both $i$ and $k$ are odd; and $\bm{Y}=\bm{J-I}$, otherwise. We can calculate that
\vspace{-0.3em}
\begin{align*}
	\text{rank}(\text{adj}(I(B_n))+\bm{D}_{n,X})_{\mathbb{Z}_2}=
	\begin{cases}
		n, & \text{if~} N=0,\\
		n+1-N, & \text{if~} N\geq1,
	\end{cases}  
\end{align*}
by Lemma \ref{rankofK_n}. Let $\mathbb{X}_1=C_{\text{index}}$. Hence, we obtain that 
\vspace{-0.6em}
\begin{align}\label{main4_eq1}
	P_{\mathbb{X}_1}=
	\begin{cases}
		z^{n}, & \text{if~} N=0,\\
		z^{n+1-N}, & \text{if~} N\geq1.
	\end{cases}  
\end{align}

\vspace{-0.1em}	
\noindent\textbf{Case 2: $\bm{X}$ satisfies exactly two of C1, C2, C3.}

\vspace{0.1em}
We discuss this case by the following two subcases. Let $\mathbb{X}_2$ and $\mathbb{X}_{2.m}$ be the sets of all $X$ that satisfy Case 2 and Subcase $2.m$, respectively, where $m\in\{1,2\}$.

\vspace{0.3em}
\noindent\textbf{Subcase 2.1: $\bm{X}$ satisfies only C1 and C2.} 
We have that
$X\subsetneqq C_{\text{index}}$. Regardless of whether there is an odd number among $i, j$ (if so, without loss of generality, assume $i$ is odd), by some elementary transformations, one has
\vspace{-0.3em}
\[\text{adj}(I(B_n))+\bm{D}_{{n,X}}
\rightarrow
\begin{bmatrix}
\bm{(J-I)}_{i,i}& \bm{O}_{i,j} & \bm{J}_{i,k} \\
\bm{O}_{j,i}& \bm{(J-I)}_{j,j} & \bm{J}_{j,k}\\ 
\bm{J}_{k,i}& \bm{J}_{k,j} & \bm{D}_{k,\overline{c}}
\end{bmatrix}
\rightarrow
\begin{bmatrix}
\bm{I}_{i,i} & \bm{O}_{i,j} & \bm{O}_{i,k} \\
\bm{O}_{j,i} & \bm{(J-I)}_{j,j} & \bm{O}_{j,k} \\ 
\bm{O}_{k,i} & \bm{O}_{k,j} & \bm{D}_{k,k^{\prime}} \\
\end{bmatrix},
\] 
where 
$k^{\prime}=\overline{c}$ if $i,j$ are both even; and $k^{\prime}=\overline{c}+1$, otherwise. It follows that



\vspace{-0.8em}
\[\text{rank}(\text{adj}(I(B_n))+\bm{D}_{n,X})_{\mathbb{Z}_2} = 2\left\lceil \frac{i+j}{2} \right\rceil+|X\cap C_{\text{index}}|,\] 
by Lemma \ref{rankofK_n}. Therefore, one has
\vspace{-0.6em}
\begin{align}\label{main4_eq2}
P_{\mathbb{X}_{2.1}}=\sum_{c=0}^{k-1}\binom{k}{c} z^{2\left\lceil \frac{i+j}{2} \right\rceil+c} = \sum\limits_{d=2}^{n}\overline{\binom{k}{d - 2\left\lceil \frac{i + j}{2} \right\rceil}}z^d.
\end{align}

\vspace{-0.3em}
\noindent\textbf{Subcase 2.2: $\bm{X}$ satisfies only C3 and one of C1 or C2.} 
If $X$ satisfies only C3 and C1, we have that
$C_{\text{index}}\subsetneqq X \subseteq B_{\text{index}}\cup C_{\text{index}}$, so $X$ contains at least one element of $B_{\text{index}}$. 

We apply elementary transformations to the full rows and columns indexed by $B_{\text{index}}$ in matrix $\text{adj}(I(B_n))+\bm{D}_{n,X}$. Specifically, we perform simultaneous swaps between a row/column whose diagonal entry is $0$ and one whose diagonal entry is $1$. This is done until the first $j-|X\cap B_{\text{index}}|$ diagonal elements are $0$ and the remaining are $1$. Thus the modified adjacency matrix is then
\[\text{adj}(I(B_n))+\bm{D}_{{n,X}}
=
\begin{bmatrix}
	\bm{J}_{i,i}-\bm{I}_{i,i} \!\!&\!\! \bm{O}_{i,j} \!\!&\!\! \bm{J}_{i,k} \\
	\bm{O}_{j,i} \!\!&\!\! \bm{J}_{j,j}-\bm{D}_{j,B_{\text{index}}\backslash (X\cap B_{\text{index}})} \!\!&\!\! \bm{J}_{j,k}\\ 
	\bm{J}_{k,i} \!\!&\!\! \bm{J}_{k,j} \!\!&\!\! \bm{I}_{k,k}
\end{bmatrix}
\rightarrow
\begin{bmatrix}
	\bm{J}_{i,i}-\bm{I}_{i,i} \!\!&\!\! \bm{O}_{i,j} \!\!&\!\! \bm{J}_{i,k} \\
	\bm{O}_{j,i} \!\!&\!\! \bm{J}_{j,j}-\bm{D}_{j,\overline{b}} \!\!&\!\! \bm{J}_{j,k}\\ 
	\bm{J}_{k,i} \!\!&\!\! \bm{J}_{k,j} \!\!&\!\! \bm{I}_{k,k}
\end{bmatrix}.
\]  


\vspace{-0.3em}
\noindent After some elementary transformations, one has
\vspace{-0.3em}
\begin{align*}
	\text{adj}(I(B_n))+\bm{D}_{{n,X}}\rightarrow
	\begin{cases}
			\bm{D}_{n,i+\overline{b}+k+1},
		& \text{if~} i+k \text{~is~even}, \\
			\bm{D}_{n,i+\overline{b}+k},
		& \text{if~} i+k \text{~is~odd}.
	\end{cases}  
\end{align*}

\vspace{-0.3em}
\noindent We can calculate that 
$\text{rank}(\text{adj}(I(B_n))+\bm{D}_{n,X})_{\mathbb{Z}_2} = 2\left\lfloor \frac{i+k}{2} \right\rfloor+1+(j-|X\cap B_{\text{index}}|).$ Thus we have 
\vspace{-0.5em}
\begin{align}\label{main4_eq3}
\sum_{C_{\text{index}}\subsetneqq X \subseteq (B\cup C)_{\text{index}}} z^{{\text{rank}(\text{adj}(I(B_n))+\bm{D}_{n,X})}_{\mathbb{Z}_2}}=\sum_{b=1}^{j}\binom{j}{b} z^{2\left\lfloor \frac{i+k}{2} \right\rfloor+1+(j-b)}=
\sum_{b=0}^{j-1}\binom{j}{b} z^{2\left\lfloor \frac{i+k}{2} \right\rfloor+1+b}.
\end{align}

\vspace{-0.5em}
\noindent By symmetry of $i,j$, if $X$ satisfies only C3 and C2, one has that
$C_{\text{index}}\subsetneqq X \subseteq A_{\text{index}}\cup C_{\text{index}}$ and
\vspace{-0.6em}
\begin{align}\label{main4_eq4}
	\sum_{C_{\text{index}}\subsetneqq X \subseteq (A\cup C)_{\text{index}}} z^{{\text{rank}(\text{adj}(I(B_n))+\bm{D}_{n,X})}_{\mathbb{Z}_2}}=\sum_{a=0}^{i-1}\binom{i}{a} z^{2\left\lfloor \frac{j+k}{2} \right\rfloor+1+a}.
\end{align}

\vspace{-0.5em}
\noindent By Equations (\ref{main4_eq3}) and  (\ref{main4_eq4}), we have that 

\vspace{-1.5em}
\begin{align}\label{main4_eq5}
	P_{\mathbb{X}_{2.2}}=\sum\limits_{d = 2}^{n}\left[\overline{\binom{i}{d-1-{2\left\lfloor \frac{j+k}{2} \right\rfloor}}} +
	\overline{\binom{j}{d-1-2\left\lfloor \frac{i+k}{2} \right\rfloor}}\right]z^d
\end{align}

\vspace{-0.3em}
Since $P_{\mathbb{X}_{2}}=P_{\mathbb{X}_{2.1}}+P_{\mathbb{X}_{2.2}}$, by Equations (\ref{main4_eq2}) and (\ref{main4_eq5}), we sum over all two subcases to obtain
\vspace{-0.2em}
\begin{align}\label{main4_eqP_2}
	P_{\mathbb{X}_2}= \sum\limits_{d=2}^{n}\left[\overline{\binom{k}{d - 2\left\lceil \frac{i + j}{2} \right\rceil}}+\overline{\binom{i}{d-1-{2\left\lfloor \frac{j+k}{2} \right\rfloor}}} +
	\overline{\binom{j}{d-1-2\left\lfloor \frac{i+k}{2} \right\rfloor}}\right]z^d.
\end{align}

\vspace{-0.1em}	
\noindent\textbf{Case 3: The set $\bm{X}$ satisfies exactly one of C1, C2 and C3.}

\vspace{0.2em}			
We discuss this case by the following three subcases. Let $\mathbb{X}_3$ and $\mathbb{X}_{3.m}$ be the sets of all subsets $X$ of $\mathbb{Z}_n$ that satisfy Case $3$ and subcase $3.m$, respectively, where $m\in\{1,2,3\}$.	

\noindent\textbf{Subcase 3.1: $\bm{X}$ satisfies only C1.} 
We have that $X\subsetneqq B_{\text{index}}\cup C_{\text{index}} \text{ and } X\cap B_{\text{index}}\neq\emptyset \text{ and } X\cap C_{\text{index}}\neq C_{\text{index}}$. By applying elementary transformations, one has
\vspace{-0.1em}
\[\text{adj}(I(B_n))+\bm{D}_{{n,X}}
\!\rightarrow\!
\begin{bmatrix}
	\bm{(J-I)}_{i,i} \!\!& \!\! \bm{O}_{i,j} \!\!& \!\! \bm{J}_{i,k} \\
	\bm{O}_{j,i} \!\!& \!\! \bm{J}_{j,j}-\bm{D}_{j,\overline{b}} \!\!& \!\! \bm{J}_{j,k}\\ 
	\bm{J}_{k,i} \!\!& \!\! \bm{J}_{k,j} \!\!& \!\! \bm{D}_{k,\overline{c}}
\end{bmatrix}
\!\rightarrow\!
	\bm{D}_{n,i+\overline{b}+\overline{c}+2}.
\]

\vspace{-0.3em}
\noindent It follows that
$\text{rank}(\text{adj}(I(B_n))+\bm{D}_{{n,X}})_{\mathbb{Z}_2}
=i+(j-|X\cap B_{\text{index}}|)+|X\cap C_{\text{index}}|+2$. 
Thus
\vspace{-0.3em}
\begin{align}\label{main4_eq6}
	P_{\mathbb{X}_{3.1}}=\sum_{b=1}^{j}\sum_{c=0}^{k-1}\binom{j}{b}\binom{k}{c} z^{i+(j-b)+c+2}=\sum_{b=0}^{j-1}\sum_{c=0}^{k-1}\binom{j}{b}\binom{k}{c} z^{i+b+c+2}.
\end{align}

\noindent\textbf{Subcase 3.2: $\bm{X}$ satisfies only C2.} 
By symmetry of $i$ and $j$, if $X$ satisfies only C2, one has 
\vspace{-0.1em}
\begin{align}\label{main4_eq7}
	P_{\mathbb{X}_{3.2}}=\sum_{a=0}^{i-1}\sum_{c=0}^{k-1}\binom{i}{a}\binom{k}{c} z^{j+a+c+2}.
\end{align}

\noindent\textbf{Subcase 3.3: $\bm{X}$ satisfies only C3.} Since $X$ satisfies only C3, we have that $C_{\text{index}}\subsetneqq X \text{ and } X\cap A_{\text{index}}\neq \emptyset \text{ and } X\cap B_{\text{index}}\neq \emptyset$. By some elementary transformations, one has
\vspace{-0.3em}
\[\text{adj}(I(B_n))+\bm{D}_{n,X}
\!\rightarrow\!
\begin{bmatrix}
	\bm{J}_{i,i}-\bm{D}_{i,\overline{a}} & \bm{O}_{i,j} & \bm{J}_{i,k} \\
	\bm{O}_{j,i} & \bm{J}_{j,j}-\bm{D}_{j,\overline{b}} & \bm{J}_{j,k}\\ 
	\bm{J}_{k,i} & \bm{J}_{k,j} & \bm{I}_{k,k}
\end{bmatrix}
\!\rightarrow\!
\bm{D}_{n,\overline{a}+\overline{b}+k+2}. 
\]


\vspace{-0.3em}
\noindent It follows that 
$\text{rank}(\text{adj}(I(B_n))+\bm{D}_{n,X})_{\mathbb{Z}_2}=(i-|X\cap A_{\text{index}}|)+(j-|X\cap B_{\text{index}}|)+k+2.$ 
Thus
\vspace{-0.6em}
\begin{align}\label{main4_eq8}
P_{\mathbb{X}_{3.3}}=\sum_{a=1}^{i}\sum_{b=1}^{j}\binom{i}{a}\binom{j}{b} z^{(i-a)+(j-b)+k+2}=\sum_{a=0}^{i-1}\sum_{b=0}^{j-1}\binom{i}{a}\binom{j}{b}z^{a+b+k+2}.
\end{align}

Since $P_{\mathbb{X}_3}\!=\!P_{\mathbb{X}_{3.1}}\!+\!P_{\mathbb{X}_{3.2}}\!+\!P_{\mathbb{X}_{3.3}}$, by Equations (\ref{main4_eq6})-(\ref{main4_eq8}), the sum over all three subcases yields: 
\begin{equation}
	\begin{aligned}\!\!\!\label{main4_eq9}
		\hspace{-0.6em}P_{\mathbb{X}_3}&=  \sum\limits_{b=0}^{j-1}\sum\limits_{c=0}^{k-1} \binom{j}{b} \binom{k}{c} z^{i+b+c+2} +
		\sum\limits_{a=0}^{i-1}\sum\limits_{b=0}^{j-1} \binom{i}{a}\binom{j}{b}z^{a+b+k+2} +  \sum_{a=0}\limits^{i-1}\sum\limits_{c=0}^{k-1} \binom{i}{a}\binom{k}{c} z^{a+j+c+2}\\
		&=\sum\limits_{d=2}^{n}\left[\sum\limits_{b=0}^{j-1}\binom{j}{b}\overline{\binom{k}{d - 2 - i - b}}\!+\!\sum\limits_{a=0}^{i-1} \binom{i}{a}\Bigg(\overline{\binom{j}{d - 2 - k -a}}\!+\!\overline{\binom{k}{d - 2 - j - a}}\Bigg)
		\right]z^d.
	\end{aligned}
\end{equation}

\vspace{0.2em}
\noindent\textbf{Case 4: The set $\bm{X}$ satisfies none of C1, C2 and C3.}

Let $\mathbb{X}_4=\{X|X \cap A_{\text{index}} \neq\emptyset \text{ and }  X \cap B_{\text{index}}\neq \emptyset \text{ and }  X\cap C_{\text{index}}\neq C_{\text{index}}\}$.
Since $X$ satisfies none of C1, C2 and C3, one has that $X\in\mathbb{X}_4$ and  
\[\text{adj}(I(B_n))+\bm{D}_{n,X}\rightarrow
\begin{bmatrix}
	\bm{J}_{i,i}-\bm{D}_{i,\overline{a}} & \bm{O}_{i,j} & \bm{J}_{i,k} \\
	\bm{O}_{j,i} & \bm{J}_{j,j}-\bm{D}_{j,\overline{b}} & \bm{J}_{j,k}\\ 
	\bm{J}_{k,i} & \bm{J}_{k,j} & \bm{D}_{k,\overline{c}}
\end{bmatrix}\rightarrow
\bm{D}_{n,\overline{a}+\overline{b}+\overline{c}+2},
\]
by some elementary transformations. 
The rank of this matrix over $\mathbb{Z}_2$ is  $\text{rank}(\text{adj}(I(B_n))+\bm{D}_{n,X})_{\mathbb{Z}_2}=(i-|X\cap A_{\text{index}}|)+(j-|X\cap B_{\text{index}}|)+|X\cap C_{\text{index}}|+2,$ 
thus
\vspace{-0.3em}
\begin{align*}
P_{\mathbb{X}_4}&=\sum_{a=1}^{i}\sum_{b=1}^{j}\sum_{c=0}^{k-1}\binom{i}{a}\binom{j}{b}\binom{k}{c} z^{(i-a)+(j-b)+c+2}=\sum_{a=0}^{i-1}\sum_{b=0}^{j-1}\sum_{c=0}^{k-1}\binom{i}{a}\binom{j}{b}\binom{k}{c} z^{a+b+c+2},
\end{align*}
which implies that
\vspace{-0.6em}
\begin{align}\label{main4_eq10}
	P_{\mathbb{X}_4}&=\sum\limits_{d=2}^{n}\sum\limits_{a = 0}^{i-1}\sum\limits_{b=0}^{j-1}\binom{i}{a} \binom{j}{b} \overline{\binom{k}{d-2-a-b}}z^d.
\end{align}

Summing over all four cases (Equations (\ref{main4_eq1}),(\ref{main4_eqP_2}),(\ref{main4_eq9}),(\ref{main4_eq10})), the partial Petrial polynomial is
\vspace{-0.3em}
 $$^{\partial}{\varepsilon_{B_n}^{\times}}(z)=\sum\limits_{X\in\mathbb{Z}_n} z^{{\text{rank}(\text{adj}(I(B_n))+\bm{D}_{n,X})}_{\mathbb{Z}_2}}=\sum\limits_{l=1}^{4}\sum\limits_{X\in\mathbb{X}_l} z^{{\text{rank}(\text{adj}(I(B_n))+\bm{D}_{n,X})}_{\mathbb{Z}_2}}= P_{\mathbb{X}_1}+P_{\mathbb{X}_2}+P_{\mathbb{X}_3}+P_{\mathbb{X}_4}.$$ Sum up the coefficients of $z$ to the power of $d$, $2\leq d\leq n$. The proof is complete.
\end{proof}

\section{Conclusion}
In this paper, we introduce an equivalent representation of the partial Petrial polynomial of bouquets from the perspective of the matrix rank, which transforms problems in topological graph theory into an algebraic framework. Furthermore, we derive a recursive formula for the partial Petrial polynomial of ribbon graphs and explicitly determine this polynomial for bouquets whose intersection graphs are cycles, where the resulting polynomial is a trinomial. 
Finally, we provide a characterization of the bouquets whose partial Petrial polynomials have the second-order term as the lowest-order term and further compute their explicit partial Petrial polynomials. In addition, we can calculate $^{\partial}{\varepsilon_{B_n}^{\times}}(z)$ if the intersection graph of $B_n$ is a complete graph (Theorem \ref{K_n}) using Lemma \ref{rankofpoly_Petrial}. In the future, we propose the following natural problem:

\begin{Problem}
Can we find an algorithm that, based on its intersection graph, determines a Euler-genus-minimizing partial Petrial graph of $B_n$\emph{?}
\end{Problem}


\section*{Acknowledgments}

\indent
This work was partially supported by the National Natural Science Foundation of China
(Nos. 12471321 and 12331013).

\noindent{\bf Data availability statement:} Data sharing not applicable to this article as no datasets were generated or analyzed during the current study.

\end{document}